\documentclass[11pt]{amsart}
\usepackage{amssymb,amsmath,amsthm,amsfonts,mathrsfs,graphicx}
\usepackage{mathtools}
\mathtoolsset{showonlyrefs}
\usepackage{comment}
\usepackage{subcaption}
\usepackage[margin=3cm]{geometry}
\usepackage[colorlinks]{hyperref}
\hypersetup{
  linkcolor=red,  % equation/section refs (\ref, \eqref)
  citecolor=red,  % citations (\cite)
  urlcolor=red    % \url, \href
}
\usepackage{cite}

\usepackage{tikz}
\usepackage{ esint }
\title[Uniform-in-diffusivity mixing]{Uniform-in-Diffusivity Mixing by Shear flows:\\ Stochastic and Dynamical Perspectives}

\author{Kyle L. Liss}
\address[Kyle L. Liss]{\newline Department of Mathematics, \ 
 University of South Carolina, 1523 Greene St., Columbia, SC 29208, USA}
\email{kliss@mailbox.sc.edu}

\author{Kunhui Luan}
\address[Kunhui Luan]{\newline Department of Mathematics, \ 
 University of South Carolina, 1523 Greene St., Columbia, SC 29208, USA}
\email{kunhui.luan@sc.edu}

\newtheorem{theorem}{Theorem}[section]
\newtheorem{lemma}[theorem]{Lemma}

\newtheorem{assumption}{Assumption}
\theoremstyle{definition}

\theoremstyle{remark}
\newtheorem{remark}{Remark}[section]

\numberwithin{equation}{section}

\def\R{\mathbb{R}}
\def\T{\mathbb{T}}

\def\pa{\partial}

\usepackage{tikz-cd}
\usepackage{soul}

\newcommand{\Z}{\mathbb{Z}}
\newcommand{\E}{\mathbf{E}}

\def\R{\mathbb{R}}
\def\N{\mathbb{N}}
\def\P{\mathbf{P}}
\def\dee{\mathrm{d}}
\def\leqc{\lesssim}
\def\pa{\partial}

\def\m{\text{m}}
\def\grad{\nabla}

\def\d{\textnormal{d}}
\numberwithin{equation}{section}

\newcommand\numberthis{\addtocounter{equation}{1}\tag{\theequation}}

\begin{document}
%%%%%%%%%%%%%%%%
\allowdisplaybreaks

\begin{abstract}
We study passive scalar mixing by parallel shear flows in the presence of weak molecular diffusion. We recover the sharp uniform-in-diffusivity mixing rate for shear flows with finitely many critical points, recently proven in \cite{RajDallas}. Our approach is based on the stochastic representation formula of the associated advection-diffusion equation and yields two short proofs. The first uses a stochastic integration-by-parts argument and gives optimal mixing under the weakest regularity assumption required in the zero-diffusion case, answering Question II in \cite[Section 4]{RajDallas}. The second adopts a dynamical systems perspective and provides a proof of shear-induced mixing that, to our knowledge, is new even in the zero-diffusivity setting.
\end{abstract}

\maketitle %\centerline{\date}

%\tableofcontents

\section{Introduction}

We consider a passive scalar $f:[0,\infty) \times \T^2 \to \R$ advected by a parallel shear flow in the presence of weak molecular diffusion. The dynamics of $f$ are governed by the advection-diffusion equation
\begin{equation} \label{eq:ADE}
\begin{cases}
    \partial_t f + b(y) \partial_x f = \nu \Delta f, \\ 
    f|_{t=0} = f_0.
\end{cases}
\end{equation}
Here, $(x,y) \in \T^2 := \R^2 / (2 \pi\Z)^2$, $b:\T \to \R$ is a smooth shear profile, and $\nu \in [0,1]$ is the diffusivity. The $x$-average of the solution is unaffected by the transport term and evolves independently according to a one-dimensional heat equation. %Since this average decouples from the remaining dynamics, 
We may therefore assume without loss of generality that the initial datum $f_0: \T^2 \to \R$ satisfies the mean-zero condition 
\begin{equation} \label{eq:meanzero}
\int_\T f_0(x,y) \, \dee x = 0 \qquad \forall \, y \in \T,
\end{equation}
which is then preserved for all $t \ge 0$.

Our interest here is in the long-time behavior of solutions to \eqref{eq:ADE}, specifically how transport by the shear flow homogenizes the scalar toward its spatial average, a process known as \textit{mixing}. The basic features of shear-induced passive scalar mixing are well understood in the absence of diffusion; see e.g. \cite{BCZ17,CZCW20}. When $\nu = 0$ and $b'(y) \neq 0$ for a.e. $y \in \T$, it can be shown that
\begin{equation} \label{eq:weak}
    f(t,x,y) \rightharpoonup 0 \hspace{0.25cm} \text{weakly in}\, \, L^2(\T^2) \, \, \text{as} \, \, t \to \infty 
\end{equation}
for every $f_0 \in L^2$ satisfying \eqref{eq:meanzero}. Since $\|f(t)\|_{L^2}$ is conserved in this case, \eqref{eq:weak} reflects a transfer of energy to increasingly fine spatial scales, driven by the horizontal stretching of scalar level sets in regions where $b'(y)$ is nonzero. 

%Since the streamline average must be removed in \eqref{eq:weak}, shear flows are not mixing in the ergodic-theoretic sense, but instead exhibit what is commonly termed \textit{phase mixing}. Parallel shears provide the simplest examples of this phenomenon, although it occurs more generally in autonomous planar flows with differential rotation along closed streamlines \cite[Theorem 4.1]{ElgindiExpository}.

The weak convergence above can be made quantitative under additional assumptions that control how the shear degenerates near the zeros of $b'$. Specifically, if $b$ has finitely many critical points and $b'$ vanishes to order at most $N$ at each of them, then
\begin{equation} \label{eq:standardmixing}
    \|f(t)\|_{H^{-1}} \leqc \langle t \rangle^{-\frac{1}{N+1}}\|f_0\|_{H^1}.
\end{equation}
The standard proof of \eqref{eq:standardmixing} is based on the explicit mode-by-mode in $x$ Fourier representation of the $\nu = 0$ solution,
\begin{equation}\label{eq:explicit}
    f(t,x,y) = \sum_{k \in \mathbb{Z} \setminus \{0\}} e^{ikx} e^{-ikb(y)t} f_k(y), \quad f_k(y) = \int_\T e^{-ikx}f_0(x,y) \, \dee x,
\end{equation}
together with decay estimates for oscillatory integrals of the form
\begin{equation} \label{eq:oscillatoryintro}
\int_\T e^{-i k b(y) t} g(y)h(y) \, \dee y, \quad g, h \in H^1(\T) 
\end{equation}
obtained via integration by parts. The decay rate above is known to be optimal in the $\nu = 0$ setting. A proof of \eqref{eq:standardmixing} can be found in \cite[Appendix A]{BCZ17}. For additional discussion, see \cite[Section 1]{RajDallas} and Section~\ref{sec:ideas} below.

In real systems, the scalar is always subject to molecular diffusion, albeit possibly weak. It is therefore natural to ask whether the mixing estimate \eqref{eq:standardmixing} persists when $\nu  > 0$. Diffusion can limit mixing in the sense that the characteristic length scale \begin{equation} \label{eq:lengthscale}
\ell(t) = \frac{\|f(t)\|_{L^2}}{\|\grad f(t)\|_{L^2}}
\end{equation}
is no longer expected, in general, to decay to zero, as in the $\nu = 0$ case \cite{MilesDoering, SamMartinTomasso}. However, decay of negative Sobolev norms captures not only reduction in length scale but also decay of scalar intensity, which suggests that estimates of the form \eqref{eq:standardmixing} may remain stable in the presence of diffusion (see also \cite[Remark 1.6]{CZ20}). At the same time, uniform-in-diffusivity estimates are generally considerably more subtle to establish than their $\nu = 0$ counterparts.

For shear flows, the first uniform-in-diffusivity mixing result since the foundational work of Kelvin on the Couette flow \cite{Kelvin} was obtained in \cite{CZ20} for a class of strictly monotone shears on $\T \times \R$. The proof combines a hypocoercivity approach \cite{Villani} with the identification of a ``good'' derivative that commutes with the transport part of \eqref{eq:ADE}. This strategy was later adapted in \cite{coti2023orientation, zelati2025nonlinear} to obtain uniform-in-diffusivity mixing, up to a relevant dissipation timescale, for kinetic transport on $\mathbb{S}^2$, where the induced shear structure exhibits degeneracies. Sharp uniform-in-diffusivity mixing for general parallel shears with finitely many critical points remained open until the recent work \cite{RajDallas}. There, the bound \eqref{eq:standardmixing} was proven for $\nu > 0$ using a resolvent representation of the solution that extends the explicit formula \eqref{eq:explicit}. The argument relies on quantitative pointwise estimates on the resolvent kernel, enabling control of the relevant oscillatory integrals with integration by parts. In the case of nondegenerate critical points, they also establish the nonzero limiting behavior of $\ell(t)$ predicted in the physics literature \cite{camassa2010analysis}.

Uniform-in-diffusivity mixing and related phenomena have also been studied beyond the passive scalar shear flow setting. For instance, vorticity mixing in the two-dimensional Euler equations, known as \textit{inviscid damping}, has been established uniformly in viscosity for the Navier-Stokes equations linearized around certain classes of spectrally stable shear flows \cite{chen2023linear, jia2023uniform}; see also \cite{beekie2026uniform}, which proves uniformity of a related damping effect known as \textit{vorticity depletion}. There are also several examples of uniform-in-diffusivity exponential mixing of passive scalars by time-dependent flows \cite{BBPSenhanced, IyerCooperman24, navarro2025exponential}.

In this paper, we recover the sharp uniform-in-diffusivity mixing result of \cite{RajDallas} for shears with finitely many critical points using the stochastic representation formula associated with \eqref{eq:ADE}. We provide two proofs, each exploiting the stochastic framework in a distinct way. The first proceeds by analyzing the stochastic analogue of \eqref{eq:explicit} and implementing an integration-by-parts argument. It reproduces the strongest possible mixing estimates available in the $\nu = 0$ setting and, in particular, answers Question II in \cite[Section 4]{RajDallas}. The second proof is more dynamical in nature and conceptually distinct from the mode-by-mode oscillatory integral analysis standard for the $\nu = 0$ problem.

\subsection{Statement of results} \label{sec:results}

Our standing assumption on the shear profile is as follows.

\begin{assumption} \label{assumption}
The function $b:\T \to \R$ is smooth and has finitely many critical points $y_1, \ldots, y_m$. Moreover, there exists $N \in \N$ such that $b'$ vanishes to order at most $N$ at each $y_i$. That is, for each $1 \le i \le m$, there exists $k \le N$ with $b^{(k+1)}(y_i) \neq 0$.
\end{assumption}

To state our first result, we introduce a norm that measures the $W^{1,1}(\T)$ regularity in $y$ of the Fourier coefficients with respect to $x$. For $g:\T^2 \to \R$, define 
$$g_k(y) = \int_\T e^{-i k x} g(x,y) \, \dee x.$$
We write $\ell_k^2(W^{1,1}_y)$ for the Banach space equipped with the norm
$$ \|g\|_{\ell_k^2 (W^{1,1}_y)} = \left(\sum_{k \in \Z}\|g_k\|^2_{W^{1,1}_y(\T)}\right)^{1/2}. $$
The continuous dual is naturally identified with $\ell_k^2 (W^{-1,\infty}_y)$. We then have the following uniform-in-diffusivity $\ell_k^2 (W^{1,1}_y) \to \ell_k^2 (W^{-1,\infty}_y)$ mixing estimate.

\begin{theorem} \label{thm:main}
Let $b: \T \to \R$ satisfy Assumption~\ref{assumption}. There exists $C > 0$, depending only on $b$, such that for all $\nu > 0$ and $f_0 \in \ell^2_k(W_y^{1,1})$ satisfying \eqref{eq:meanzero}, the associated smooth solution of \eqref{eq:ADE} satisfies
    \begin{equation} \label{eq:main}
        \|f(t)\|_{\ell_k^2 (W_{y}^{-1,\infty})} \le C \langle t \rangle^{-\frac{1}{N+1}}\|f_0\|_{\ell_k^2 (W_y^{1,1})}.
    \end{equation}
\end{theorem}

\begin{remark}
   Since $H^1(\T) \hookrightarrow W^{1,1}(\T)$ and $\ell_k^2 H^1_y = L^2_x H^1_y$, we have the continuous embeddings
   $$ L^2_x H^1_y \hookrightarrow \ell_k^2 W_y^{1,1} \quad \text{and} \quad \ell_k^2 W_y^{-1,\infty} \hookrightarrow L^2_x H^{-1}_y. $$
   Therefore, \eqref{eq:main} implies the uniform-in-diffusivity $L^2_x H^1_y \to L^2_x H^{-1}_y$ mixing estimate obtained in \cite[Theorem 1.2]{RajDallas}. Our use of $W^{1,1}_y$ matches the weakest regularity required to yield $\langle t \rangle^{-1/(N+1)}$ decay in the $\nu = 0$ setting and answers Question II in \cite[Section 4]{RajDallas}.     
\end{remark}

Theorem~\ref{thm:main} above, as well as the classical $\nu = 0$ estimates obtained via integration by parts, measure mixing through norms that average negative regularity in the $y$ variable across the shear direction $x$, typically in an $L_x^2$ sense. Our second, more dynamical, proof of uniform-in-$\nu$ mixing naturally yields decay in a genuinely different topology, in which negative $y$-regularity is controlled uniformly in $x$. Precisely, we obtain the following estimate.

\begin{theorem} \label{thm:main2}
Let $b:\T \to \R$ satisfy Assumption~\ref{assumption}. There exists $C > 0$, depending only on $b$, such that all $\nu > 0$ and $f_0 \in L^\infty_x (W^{1,\infty}_y)$ there holds 
    \begin{equation}
        \|f(t)\|_{L^\infty_x (W_y^{-1,1})} \le C \langle t \rangle^{-\frac{1}{N+1}}\|f_0\|_{L^\infty_x (W^{1,\infty}_y)}.
    \end{equation}
\end{theorem}

\begin{remark}
    The integration by parts based argument used to prove Theorem~\ref{thm:main} can also give decay of $\|f(t)\|_{L^\infty_x (W_y^{-1,1})}$, but at the cost of paying more than one derivative on the initial data, which is not optimal.
\end{remark}

\subsection{Discussion of the proof} \label{sec:ideas}

The starting point for the proofs of both Theorem~\ref{thm:main} and~\ref{thm:main2} is to write the solution of \eqref{eq:ADE} using its stochastic representation formula. Let $W_t$ and $B_t$ be independent standard Brownian motions on $\R$, defined on a common probability space $(\Omega, \mathcal{F}, \P)$. We identify $\T = \R / 2\pi \Z$ with the fundamental domain $[0,2\pi)$ and with a slight abuse of notation write $b:\R \to \R$ for the $2\pi$-periodic extension of the original shear profile. Define the random flow map $\Phi_t(x,y) = \Pi(x_t,y_t)$, where $\Pi:\R^2 \to \T^2$ denotes the canonical projection and $(x_t,y_t) \in \R^2$ solves the stochastic differential equation
\begin{equation} \label{eq:SDE}
    \begin{cases}
        \dee x_t = -b(y_t) \, \dee t + \sqrt{2\nu} \, \dee B_t \\ 
        \dee y_t = \sqrt{2 \nu} \, \dee W_t
    \end{cases}
\end{equation}
with initial condition $(x,y) \in \T^2$. Then,
\begin{equation} \label{eq:flowmap}
    \Phi_t(x,y) = \Pi\left(x + \sqrt{2 \nu}B_t - \int_0^t b(y+\sqrt{2\nu}W_s) \, \dee s, \,
        y + \sqrt{2\nu}W_t\right)
\end{equation}
and the solution of \eqref{eq:ADE} is given by 
\begin{equation} \label{eq:stochasticrep}
    f(t,x,y) = \E [f_0 \circ \Phi_t(x,y)],
\end{equation}
where $\E$ denotes expectation with respect to $\P$.

Our proof of Theorem~\ref{thm:main} proceeds by a stochastic version of the integration-by-parts argument in \cite[Appendix A]{BCZ17}. As we show in Section~\ref{sec:IBP}, \eqref{eq:flowmap} and \eqref{eq:stochasticrep} together with standard computations in the $\nu = 0$ setting reduce the proof of Theorem~\ref{thm:main} to the estimate
\begin{equation}
    \E\left|\int_\T e^{-ik \varphi_t(y)} g(y) h(y) \, \dee y\right| \leqc \langle t \rangle^{-\frac{1}{N+1}}\|g\|_{W^{1,1}(\T)}\|h\|_{W^{1,1}(\T)}
\end{equation}
for $k \in \Z \setminus \{0\}$, where $\varphi_t(y)$ is now the random phase function
\begin{equation} \label{eq:randomphase}
\varphi_t(y) = \int_0^t b(y + \sqrt{2\nu}W_s)\, \dee s. 
\end{equation}
One should interpret $\varphi_t$ as the stochastic analogue of the phase $t b(y)$ in \eqref{eq:oscillatoryintro}. 

In the deterministic setting, one obtains \eqref{eq:standardmixing} by decomposing the integral in \eqref{eq:oscillatoryintro} as 
\begin{equation} \label{eq:standarddecomp}
\begin{aligned}
    \int_\T e^{-iktb(y)} g(y) h(y) \, \dee y & = \sum_{i=1}^m \int_{|y-y_i|\le \delta(t)} e^{-ik t b(y)} g(y) h(y) \, \dee y \\ 
    & \quad + \int_{\mathrm{complement}} \frac{1}{i k t b'(y)} \partial_y e^{-ik t b(y)} g(y) h(y) \, \dee y,
\end{aligned}
\end{equation}
where $\{y_i\}_{i=1}^m$ denote the zeros of $b'$ and $\delta(t)$ is a time-dependent length scale. On the complement, $|t b'(y)| \gtrsim t \delta(t)^N$. The decay rate $\langle t \rangle^{-1/(N+1)}$ then arises after choosing $\delta(t) = t^{-1/(N+1)}$ to optimize the gain from integration by parts with the $\mathcal{O}(\delta(t))$ contribution of the degenerate region.

It is not immediately clear that the same argument can be carried out with the stochastic phase $\varphi_t(y)$. Unlike the deterministic quantity $tb'(y)$, the derivative 
\begin{equation}
    \varphi_t'(y) = \int_0^t b'(y+\sqrt{2\nu} W_s) \, \dee s
\end{equation}
samples $b'$ along an entire random trajectory. In particular, points initially near a critical point may drift through regions where $b'$ changes sign or varies significantly in magnitude, and the resulting cancellations in the time integral obscure the behavior of $\varphi_t'(y)$. The key observation behind our proof is that, except on rare events that can be treated as an error, a decomposition analogous to \eqref{eq:standarddecomp} can still be implemented. Specifically, we show in Lemma~\ref{key_lemma} that for times $t \ll \nu^{-1}$, one can 
partition $\T$ into a collection of small random intervals, with total measure compatible with the optimal choice $\delta(t)$ above, and a complementary region on which \begin{equation} \label{eq:lbd}
|\varphi_t'(y)| \gtrsim t \delta(t)^N = t^{\frac{1}{N+1}},
\end{equation} 
in direct analogy with the lower bound for $t|b'(y)|$ in the deterministic setting. This allows us to recover the optimal $\nu = 0$ decay rate and regularity in Theorem~\ref{thm:main} with a stochastic version of \eqref{eq:standarddecomp}.

Theorem~\ref{thm:main2} relies on the same key stochastic Lemma~\ref{key_lemma}, but exploits it in a conceptually different way. Rather than proceeding mode-by-mode, the proof is based on understanding the geometry of images of vertical segments under the random flow map. The basic idea of the argument is as follows. Fix $x \in \T$ and consider the vertical segment $I = \{(x,y) \in \T^2: y \in \T\}$. By \eqref{eq:flowmap} and a change of variables, for any test function $g: \T^2 \to \R$ we may write
\begin{equation} \label{eq:introdyamics} \E \int_\T f_0 \circ \Phi_t(x,y) g(x,y) \, \dee y = \E\int_{\Phi_t(I)} J_t(y) f_0(x,y) (g \circ \Phi_t^{-1})(x,y) \, \dee \m_s, 
\end{equation}
where $J_t$ is the associated Jacobian factor and $\m_s$ denotes the arclength measure along the random curve $\Phi_t(I)$. When parametrized by the initial vertical coordinate $y$, the tangent vector to $\Phi_t(I)$ is $(-\varphi_t'(y),1)$, so that the slope of $\Phi_t(I)$ is locally determined by $1/|\varphi_t'(y)|$. Consequently, from the lower bound \eqref{eq:lbd} provided by Lemma~\ref{key_lemma}, the image $\Phi_t(I)$ is nearly horizontal with slope bounded by $t^{-1/(N+1)}$, except on a small random subset analogous to the intervals removed in \eqref{eq:standarddecomp}. Along any nearly horizontal piece $\gamma \subseteq \Phi_t(I)$ that crosses $\T^2$ once horizontally, the mean-zero condition \eqref{eq:meanzero} yields 
\begin{equation} \label{eq:gammamean}
    \left|\int_\gamma f_0(x,y) \, \dee m_s\right| \leqc \sup_{(x,y) \in \gamma}|\varphi_t'(y)|^{-1} \|\partial_y f_0\|_{L^\infty} \leqc \langle t \rangle^{-\frac{1}{N+1}} \|\partial_y f_0\|_{L^\infty}.
\end{equation}
It can be shown that the factors $J_t$ and $g \circ \Phi_t^{-1}$ vary slowly along $\gamma$ and may be treated as constants, up to controlled errors. We obtain decay of \eqref{eq:introdyamics} by decomposing $\Phi_t(I)$ into a family $\{\gamma_i\}$ of such curves and applying \eqref{eq:gammamean}. 

 The proof of Theorem~\ref{thm:main2} is perhaps more interesting than the result itself. It provides an argument for shear-induced mixing that we believe is new, even in the $\nu = 0$ case. Although the integration-by-parts proof of uniform mixing is somewhat simpler in the present setting, the dynamical argument is based on a more general principle, namely that phase mixing arises from the stretching of curves under the flow map. It may be an interesting direction of future research to investigate uniform-in-diffusivity mixing by more general two-dimensional autonomous flows from this perspective.

  The control of $\varphi_t'(y)$ given by Lemma~\ref{key_lemma} holds only on a time interval $t \in [0,t_\nu]$, where $t_\nu \ll \nu^{-1}$. For times $t \gtrsim t_\nu$, we rely on decay in strong norms such as $L^2$ occurring on timescales much faster than the diffusive timescale $\nu^{-1}$. This phenomenon, known as \textit{enhanced dissipation}, arises from the interaction between transport-induced small-scale formation and diffusion. Starting with the seminal paper \cite{CKRZ08}, enhanced dissipation has been studied extensively in a wide range of settings; see, e.g. \cite{Vukadinovic21, BCZE22, ECZM19, FengIyer19, wei2021diffusion, ELM23, BBPSenhanced}. In the shear flow context, representative works include \cite{VictorJonathanKyle,Villringer24, He22, ABN22, BCZ17, CZD20, CZG21}. 

\begin{comment}
parallel and radial shears: \cite{VictorJonathanKyle,Villringer24, He22, ABN22, BCZ17, CZD20, CZDrivas19, CZG21}

2d Hamiltonian: \cite{Vukadinovic21, DolceHamiltonian24, BCZE22}

General theory: \cite{CKRZ08, ECZM19, FengIyer19, wei2021diffusion}

log nu decay: \cite{ELM23, BBPSenhanced, IyerCooperman24}
\end{comment}

\subsection*{Acknowledgments}
The research of Kunhui Luan has been partially supported by NSF grant DMS-2238219 and George Johnson Fellowships. The authors thank Raj Beekie for helpful discussions and feedback on a first draft. 

\section{Key stochastic lemmas}

The main goal of this section is to establish two lemmas concerning the behavior of 
\begin{equation}\label{eq:randomderivative}
S_t(y):=\varphi_t'(y) = \int_0^t b'(y+\sqrt{2\nu}W_s) \, \dee s
\end{equation}
that will play a key role in the proofs of Theorems~\ref{thm:main} and~\ref{thm:main2}. 

We begin with some notation. The function $S_t$ depends on the noise realization $\omega \in \Omega$. When we wish to indicate this explicitly, we write $S_t(y,\omega)$, but otherwise suppress the dependence on $\omega$. We continue to write $b:\R \to \R$ for the $2\pi$-periodic extension of the shear profile, so that $S_t(y)$ is defined for all $y \in \R$. Let 
\begin{equation}\label{eq:distance}
    d_\T(y,y') = \min_{k \in \Z}|y-y'-2\pi k|, \qquad y,y' \in \R,
\end{equation}
which induces the standard geodesic distance on $\T = [0,2\pi)$. For $y \in \T$ and $r > 0$, we write 
$$B_r(y) = \{y \in \T: d_\T(y,y') < r\} \subseteq \T$$ to denote the standard metric ball on $\T$.

let $p > 0$ satisfy
$$ \frac{N+1}{N+3} < p < 1, $$
and define the timescale $t_\nu = \nu^{-p}$. This scale is much longer than the enhanced dissipation time for \eqref{eq:ADE}, but much shorter than the $\nu^{-1}$ diffusion timescale of the heat equation. For $0 < \delta \ll 1$, define 
\begin{equation} \label{eq:goodnoise}
\Omega_\delta := \{\omega \in \Omega: \sup_{0 \le t \le t_\nu} |\sqrt{2\nu}W_t| \le \delta\}.
\end{equation}
Throughout the paper, $\Omega_\delta$ plays the role of a set of ``good'' noise realizations. The fact that $\Omega \setminus \Omega_\delta$ is a rare event which can be treated as an error is described by the following lemma.
\begin{lemma}\label{lem:badnoise}
    For any $\delta,q>0$, there exists $\nu_0>0$ such that 
    \begin{equation}\label{eq:goodnoisebound}
    \P(\Omega\setminus \Omega_\delta) \le \nu^q, \quad \forall \, \nu \in (0,\nu_0].
\end{equation}
In particular, for $\nu$ sufficiently small we have
\begin{equation}\label{eq:badnoisebound}
\P(\Omega\setminus \Omega_\delta) \le t^{-\frac{1}{N+1}}, \quad \forall \, t\in [1,t_\nu]. 
\end{equation}
\end{lemma}
\begin{proof}
By the symmetry of Brownian motion and the reflection principle, we have
    \begin{align*}
    \P\left(\sup_{0 \le t \le t_\nu} |\sqrt{2\nu}W_t| > \delta\right) = 2\P\left(\sup_{0 \le t \le t_\nu} \sqrt{2\nu}W_t > \delta\right)&= 4\P(\sqrt{2\nu}W_{t_\nu}>\delta) = 4 \P(W_{2\nu t_\nu} > \delta). 
\end{align*}
Thus,
\begin{equation}
   \P\left(\sup_{0 \le t \le t_\nu} |\sqrt{2\nu}W_t| > \delta\right) = \frac{4}{\sqrt{4\pi \nu t_\nu}} \int_{\delta}^\infty \exp\left(-\frac{z^2}{4\nu t_\nu}\right) \, \dee z \le C \exp\left(-\frac{\delta^2}{8 \nu^{1-p}}\right),
\end{equation}
where $C> 0$ depends only on $\delta$. The first claim in the lemma follows immediately. We then obtain the second by choosing $q > p/(N+1)$.

%Here, we use the scaling of Brownian motion $\sqrt{2\nu}W_s \stackrel{d}{=} W_{2\nu s}$ and the definition $t_\nu = \nu^{-p}$. Observe the above probability decays exponentially fast to zero as $\nu \to 0$, so in particular, we can find some $\nu_0>0$ such that the following holds for all $\nu \in (0,\nu_0]$.
%\[
%4\frac{\nu^{\frac{1-p}{2}}}{\delta\sqrt{\pi}} \exp \bigg(-\frac{\delta^2}{4\nu^{1-p}}\bigg)\leq \nu^q.
%\]
%Next, since $t\leq t_\nu \leq \nu^{-p}$, we get $t^{-\frac{1}{N+1}}\geq \nu^{\frac{p}{N+1}}$. The second claim \eqref{eq:badnoisebound} follows once we pick $q\geq \frac{p}{N+1}$, since $\nu$ is small. 
\end{proof}

For $\omega \in \Omega$, $c > 0$, and $t > 0$, let 
\begin{equation} \label{eq:Atdef}
    A_{t,\omega}(c) = \{y \in \T = [0,2\pi): |S_t(y,\omega)| \le c t^{\tfrac{1}{N+1}}\}.
\end{equation}
The lemma below provides an estimate on the measure of $A_{t,\omega}(c)$ for $\omega \in \Omega_\delta$ and $1 \le t \le t_\nu$. In the statement and throughout the remainder of the paper, we write $a_1 \leqc a_2$ when $a_1 \le C a_2$ for a constant $C > 0$ that is independent of $\nu \in (0,1]$, $t > 0$, and $\omega \in \Omega$, but depends possibly on the shear $b$ and $\delta > 0$ from \eqref{eq:goodnoise}. Recall also from Assumption~\ref{assumption} that $\{y_i\}_{i=1}^m$ denotes the set of critical points of $b$.

\begin{lemma}\label{key_lemma}

Let $\delta > 0$ be sufficiently small, depending only on $b$. Then, there exists $c > 0$ such that for every $\omega \in \Omega_\delta$, $1 \le t \le t_\nu$, and $\nu \in (0,1]$, we have 
\begin{equation} \label{eq:keylemma}
    A_{t,\omega}(c) \subseteq \bigcup_{i=1}^m I_i,
\end{equation}
where $\{I_i\}_{i=1}^m$ is a family of disjoint intervals on $\T$ satisfying 
\begin{equation}
    \sum_{i=1}^m |I_i| \leqc t^{-\frac{1}{N+1}}.
\end{equation}
\end{lemma}

\begin{proof}
We assume at least that 
$$ \delta < \frac{1}{4} \min_{i \neq j}d_{\T}(y_i,y_j),$$
which ensures that $B_{2\delta}(y_i) \cap B_{2\delta}(y_j) = \emptyset$ for $i \neq j$.

\textbf{Step 1}: We first claim that 
\begin{equation} \label{eq:step1goal}
A_{t,\omega}(c) \subseteq \bigcup_{i=1}^m B_{2\delta}(y_i)
\end{equation}
provided that $c$ is sufficiently small. Let $\mathcal{M} = \{y_i\}_{i = 1}^m$ and fix $y \in \T$ with $d_\T(y,\mathcal{M})\geq 2\delta$. To prove \eqref{eq:step1goal}, we must show that $y\not\in A_{t,\omega}(c)$. From the definition of $\Omega_\delta$, we have
    \[
    \inf_{s \in [0,t]} d_\T(y+\sqrt{2\nu}W_s, \mathcal{M}) \ge d_\T(y,\mathcal{M}) - \sup_{s \in [0,t]} |\sqrt{2\nu} W_s| \ge \delta.
    \]
    Therefore, $b'(y+\sqrt{2\nu} W_s)$ has a fixed sign for $s \in [0,t]$ and 
    $$ \inf_{s \in [0,t]} |b'(y+\sqrt{2\nu} W_s)| \gtrsim 1. $$
    It follows that 
    $$ |S_t(y)| = \left|\int_0^t b'(y+\sqrt{2\nu}W_s) \, \dee s\right| = \int_0^t |b'(y+\sqrt{2\nu}W_s)| \, \dee s \gtrsim t \ge t^{\frac{1}{N+1}}. $$
    This implies $y \not \in A_{t,\omega}(c)$ for $c$ sufficiently small, completing the proof of \eqref{eq:step1goal}.
     
\textbf{Step 2}: For each critical point $y_i \in [0,2\pi)$, let $J_i = (y_i - 2\delta, y_i + 2\delta) \subseteq \R$ and define 
     \begin{equation} \label{eq:Aidef}
         A_{t,\omega}^i(c) = \{y\in J_i: |S_t(y)|\le ct^{\frac{1}{N_i+1}}\} \subseteq \R,
     \end{equation}
where $1\le N_i \le N$ denotes the order of vanishing of $b'$ at $y_i$. Observe that 
\begin{equation} \label{eq:Jiprojection}
    \pi(J_i) = B_{2\delta}(y_i),
\end{equation}
where $\pi: \R \to \T$ denotes the canonical projection $\pi(y) = y \mod 2\pi$. Since $b'$ is $2\pi$-periodic we have $S_t(y) = S_t(\pi(y))$ for any $y \in \R$, and it follows from \eqref{eq:Jiprojection} that
\begin{equation} \label{eq:Aiinclusion}
    A_{t,\omega}(c) \cap B_{2\delta}(y_i) \subseteq \pi(A_{t,\omega}^i(c)).
\end{equation}
Here, we have noted also that $N_i \le N$ and the inclusion above would be equality if $N_i$ were replaced by $N$ in \eqref{eq:Aidef}.

In view of Step 1 and \eqref{eq:Aiinclusion}, to complete the proof it is sufficient to show that for each $1 \le i \le m$, there exists an interval $\hat{I}_i \subseteq \R$ such that 
\begin{equation}\label{eq:step2goal}
         A_{t,\omega}^i(c) \subseteq \hat{I}_i \quad \text{and} \quad |\hat{I}_i| \leqc t^{-\frac{1}{N_i+1}}.
     \end{equation}
The intervals $I_i$ in the lemma statement are then given by $\pi(\hat{I}_i)$. We split the proof of \eqref{eq:step2goal} into cases depending on the parity of $N_i$.

\textit{Case 1}: First, assume that $N_i$ is even and suppose without loss of generality that $b^{(N_i+1)}(y_i) > 0$. Then, by expanding $b'$ around $y_i$ to order $N_i$ using Taylor's theorem, it follows that $b'(y) \ge 0$ near $y_i$ and
\begin{equation} \label{eq:evenlower}
   b'(y) \gtrsim (y-y_i)^{N_i} \quad \text{for} \quad |y-y_i| \le 3 \delta,
\end{equation}
provided that $\delta$ is taken sufficiently small depending on $b$. Fix $y \in A_{t,\omega}^i(c) \subseteq J_i$. By \eqref{eq:evenlower} and the definitions of $\Omega_\delta$ and $A^i_{t,\omega}(c)$, we have
\begin{equation}\label{eq:even_St}
    c t^{\frac{1}{N_i+1}} \ge S_t(y) = \int_0^t b'(y+\sqrt{2\nu}W_s) \, \d s \gtrsim \int_0^t (y+\sqrt{2\nu} W_s-y_i)^{N_i}\, \d s.
\end{equation}
By Hölder's inequality,
   \[
   \bigg|\int_0^t (y+\sqrt{2\nu} \ W_s-y_i) \, \d s \bigg|^{N_i} \leq t^{N_i-1} \int_0^t (y+\sqrt{2\nu}\  W_s-y_i)^{N_i} \, \d s, \]
which after defining $a = \int_0^t W_s \, \d s \in \R$ becomes 
\begin{equation} \label{eq:Holder2}
    t\,|y-t^{-1}\sqrt{2\nu} \, a - y_i|^{N_i} \le \int_0^t (y+\sqrt{2\nu}\  W_s-y_i)^{N_i} \, \d s. 
\end{equation}
Putting \eqref{eq:Holder2} into \eqref{eq:even_St} yields
$$\left|y+t^{-1}\sqrt{2\nu}\,a-y_i\right|\lesssim t^{-\frac{1}{N_i+1}}.$$
Therefore, defining $\bar{y} = y_i-t^{-1}\sqrt{2\nu}\, a$, we have shown that $$y \in \left(\bar{y} - Ct^{-\frac{1}{N_i+1}},\bar{y} + Ct^{-\frac{1}{N_i+1}}\right):=I$$ for some constant $C>0$ depending only on $b$ and $\delta$. This proves \eqref{eq:step2goal} with $\hat{I}_i = I \cap J_i$.

\textit{Case 2}: We now suppose that $N_i$ is odd. Then, $b''$ vanishes to even order at $y_i$ and we may suppose without loss of generality that $b^{(N_i+1)}(y_i) > 0$. Then, as in \eqref{eq:evenlower}, it follows from Taylor's theorem that
\begin{equation} \label{eq:oddlower}
    b''(y) \gtrsim (y-y_i)^{N_i-1} \quad \text{for} \quad |y-y_i| \le 3 \delta
\end{equation}
whenever $\delta$ is sufficiently small. Fix $y, \bar{y} \in A^i_{t,\omega}(c)$ with $y > \bar{y}$. To prove \eqref{eq:step2goal}, it is sufficient to show that 
\begin{equation} \label{eq:step2bgoal}
    y-\bar{y} \leqc t^{-\frac{1}{N_i+1}}.
\end{equation}
By \eqref{eq:oddlower} and the definition of $\Omega_\delta$, for any $s \in [0,t]$ we have
   \begin{align*}
         b'(y+\sqrt{2\nu}W_s)-b'(\bar{y}+\sqrt{2\nu}W_s) &= \int_{\bar{y}}^{y} b''(z+\sqrt{2\nu}W_s)\,\d z \\ &\gtrsim \int_{\bar{y}}^{y} (z+\sqrt{2\nu}W_s-y_i)^{N_i-1} \, \d z\\ &\gtrsim (y+\sqrt{2\nu}W_s-y_i)^{N_i} - (\bar{y}+\sqrt{2\nu}W_s-y_i)^{N_i} \\
         &\gtrsim (y-\bar{y})^{N_i}.
   \end{align*}
  In the final line we have applied the elementary inequality $x^{k}-y^{k} \ge c_k (x-y)^{k}$, which holds for any $x>y$ and odd $k$. Then, from the definition of $A_{t,\omega}^i(c)$ we obtain
  \begin{equation}\label{eq:odd_bound}
      2 c t^{\frac{1}{N_i+1}}\ge S_t(y) - S_t(\bar{y}) = \int_0^t b'(y+\sqrt{2\nu}W_s)-b'(\bar{y}+\sqrt{2\nu}W_s) \, \d s \gtrsim t (y-\bar{y})^{N_i}. 
  \end{equation}
  This implies \eqref{eq:step2bgoal} and completes the proof.
\end{proof}

\begin{remark}\label{rmk:zeros}
The argument in Step 2 above also characterizes the zero set of $S_t$ in the vicinity of the critical points. In particular, for $\delta > 0$ sufficiently small, $\omega \in \Omega_\delta$, and $1 \le t \le t_\nu$ there holds
\begin{equation}\label{eq:zeros}
    \#\{y \in B_{2\delta}(y_i): S_t(y,\omega) = 0\} \le 1, \quad 1 \le i \le m.
\end{equation}
Indeed, when $b'$ vanishes to even order at $y_i$, \eqref{eq:even_St} shows that $S_t(y) = 0$ only if $y = y_i$ and $W_s = 0$ for every $s \in [0, t]$, which implies that $S_t(y)$ is almost surely non-vanishing on $B_{2\delta}(y_i)$. In the odd order case, it is an immediate consequence of \eqref{eq:odd_bound} that $S_t(y) - S_t(\bar{y}) > 0$ for distinct points $y,\bar{y} \in B_{2\delta}(y_i)$, and hence $S_t(y)$ can have at most one zero in $B_{2\delta}(y_i)$. Notice moreover that
\[
S_t'(y)=\int_0^t b''(y+\sqrt{2\nu}W_s)\, \dee s
\]
has the same structure as \(S_t\), with \(b'\) replaced by \(b''\). If \(b'\) vanishes to order \(N_i\) at \(y_i\), then either \(b''(y_i)\neq 0\), in which case \(b''\) has no zeros in a sufficiently small neighborhood of \(y_i\), or \(b''\) vanishes to finite order \(N_i-1\) at \(y_i\). Therefore, the argument of Step 2 can also be used to show that \eqref{eq:zeros} holds with \(S_t\) replaced by \(S_t'\), after possibly shrinking \(\delta\).
\end{remark}

The next lemma will be needed to control error terms that arise in the proofs of Theorems~\ref{thm:main} and~\ref{thm:main2}. 

\begin{lemma}\label{key_lemma_2} Let $\delta > 0$ be sufficiently small and $c>0$ be the associated constant guaranteed by Lemma~\ref{key_lemma}. Fix $t \in [1,t_\nu]$ and $\omega \in \Omega_\delta$, and define
$$ J = \T \setminus \cup_{i=1}^m I_i, $$
where $\{I_i\}_{i=1}^m$ is as in \eqref{eq:keylemma}. Then,
    \[
  \int_J\bigg|\frac{\dee}{\dee y}\left(\frac{1}{S_t(y)}\right)\bigg| \, \dee y \lesssim t^{-\frac{1}{N+1}}.
    \]
\end{lemma}

\begin{proof}

We split the integral into two terms:
\begin{align*}
    \int_J\bigg|\frac{\dee}{\dee y}\left(\frac{1}{S_t(y)}\right)\bigg| \, \dee y & = \int_{J\cap (\cup_{i}B_{2\delta}(y_i))^c}\frac{|S_t'(y)|}{S_t^2(y)} \,  \dee y +\sum_{i}\int_{J\cap B_{2\delta}(y_i)}\bigg|\frac{\dee}{\dee y}\left(\frac{1}{S_t(y)}\right)\bigg| \, \dee y \\
    &:= T_1 + T_2.
\end{align*}
For the estimate of $T_1$, we first observe as in Step 1 of the proof of Lemma~\ref{key_lemma} that $|S_t(y)|\gtrsim t$ for $y \in J \cap (\cup_i B_{2\delta}(y_i))^c$. Therefore,
$$ T_1 \leqc \frac{\|S_t'\|_{L^\infty}}{t^2} \le \frac{t\|b''\|_{L^\infty}}{t^2} \le t^{-1} \le t^{-\frac{1}{N+1}}. $$
We now turn to the second term. By Remark~\ref{rmk:zeros}, $S_t''$ has at most one zero in each $B_\delta(y_i)$, provided that $\delta$ is sufficiently small. Thus, for every $1 \le i \le m$ there exists a family of intervals $\{I_{i,j}\}_{j=1}^3$ such that $$J \cap B_{2\delta}(y_i) = \bigcup_{j=1}^3 I_{i,j}$$ and $S_t'$ has a fixed sign on each $I_{i,j}$. There are only three intervals in the decomposition above since only $I_i$ intersects $B_{2\delta}(y_i)$. For any $I_{i,j}$, we have 
$$ \int_{I_{i,j}} \bigg|\frac{\dee}{\dee y}\left(\frac{1}{S_t(y)}\right)\bigg| \, \dee y = \left|\int_{I_{i,j}} \frac{\dee}{\dee y}\frac{1}{S_t(y)} \, \dee y \right| \leqc t^{-\frac{1}{N+1}}, $$
where in the inequality we applied the fundamental theorem of calculus and the fact that $|S_t(y)| \gtrsim t^{\frac{1}{N+1}}$ for $y \in J$. Summing the inequality above over $1 \le i \le m$ and $1 \le j \le 3$ establishes that $T_2 \leqc t^{-\frac{1}{N+1}}$, which concludes the proof.
\end{proof}

\section{Stochastic Integration by parts: Proof of Theorem~\ref{thm:main}} \label{sec:IBP}
In this section, we prove Theorem \ref{thm:main}. In view of \eqref{eq:stochasticrep}, our goal is to show that for any initial datum $f_0\in \ell_k^2 (W_y^{1,1})$ satisfying \eqref{eq:meanzero} and test function $g \in \ell_k^2 (W_y^{1,1})$ there holds
\begin{equation}\label{eq:IBP_goal}
\int_{\T^2}\E[f_{0}\circ \Phi_{t}(x,y)] \,  g(x,y)  \, \dee x \, \dee y \leqc \langle t \rangle^{-\frac{1}{N+1}}\|f_0\|_{\ell_k^2 (W_y^{1,1})} \|g\|_{\ell_k^2 (W_y^{1,1})}.
\end{equation}

The key step in the proof is applying Lemmas~\ref{key_lemma} and~\ref{key_lemma_2} to deduce the following non-stationary phase type estimate pointwise in $\omega \in \Omega_\delta$. Recall that $\varphi_t$ denotes the random phase function defined in \eqref{eq:randomphase}.

\begin{lemma} \label{lem:IBPkey}
    Let $\delta>0$ be sufficiently small and $\omega \in \Omega_\delta$. For every $F, g \in W^{1,1}(\T)$, $k \in \Z \setminus \{0\}$, and $1 \le t \le t_\nu$ there holds 
    $$ \left|\int_\T e^{-i k \varphi_t(y)} F(y) g(y) \, \dee y \right| \leqc t^{-\frac{1}{N+1}} \|F\|_{W^{1,1}} \|g\|_{W^{1,1}}. $$
\end{lemma}

\begin{proof}
Let $\{I_j\}_{j=1}^m$ denote the collection of intervals guaranteed by Lemma~\ref{key_lemma} and define $J = \T \setminus \cup_{j=1}^m I_j$. We have 
\begin{equation} \label{eq:IBPsplit}
   \left|\int_\T e^{-i k \varphi_t(y)} F(y) g(y) \, \dee y \right| \le \sum_{j=1}^m\left|\int_{I_j}e^{-i k \varphi_t(y)} F(y) g(y) \, \dee y \right| + \left|\int_J e^{-i k \varphi_t(y)} F(y) g(y) \, \dee y \right|.
\end{equation}
For the estimate of the first term above, we apply Lemma~\ref{key_lemma} and the Sobolev embedding $W^{1,1}(\T) \hookrightarrow L^\infty(\T)$ to obtain
\begin{equation}\label{eq:IBPterm1}
\sum_{j=1}^m\left|\int_{I_j}e^{-i k \varphi_t(y)} F(y) g(y) \, \dee y \right| \le \|F\|_{L^\infty} \|g\|_{L^\infty} \sum_{j=1}^m |I_j| \leqc t^{-\frac{1}{N+1}}\|F\|_{W^{1,1}}\|g\|_{W^{1,1}}.
\end{equation}
To bound the integral over $J$, we begin by using the fact that $ \varphi_t'(y) = S_t(y)$ to rewrite the term as
\begin{equation}\label{eq:stationary phase}
     \left|\int_{J}e^{-ik\varphi_t(y)} F(y)g(y) \, \d y\right| = \left|\frac{i}{k} \int_{J} \frac{1}{S_t(y)}\pa_y (e^{-ik\varphi_t(y)})F(y)g(y) \, \d y\right|.
\end{equation}
Integration by parts yields
\begin{align}\label{eq:IBPthreeterms}
   \bigg|\frac{i}{k}& \int_{J}\frac{1}{S_t(y)}\pa_y (e^{-ik\varphi_t(y)})F(y)g(y) \, \d y \bigg|= \bigg|\frac{i}{k} (B+T_1+T_2)\bigg|,
\end{align}
where $B$ collects the boundary terms and $T_1$, $T_2$ are defined as 
\begin{equation}
     T_1 := -\int_{J} \frac{\dee}{\dee y} \left(\frac{1}{S_t(y)}\right) e^{-ik\varphi_t(y)}F(y)g(y) \, \dee y, \quad
    T_2 := -\int_{J}  \frac{1}{S_t(y)} e^{-ik\varphi_t(y)}\pa_y(F(y)g(y))\, \dee y.
\end{equation}
Since Lemma~\ref{key_lemma} guarantees that $|S_t(y)| \gtrsim t^{\frac{1}{N+1}}$ for $y \in J$, we have 
$$|B| \leqc t^{-\frac{1}{N+1}} \|F\|_{L^\infty} \|g\|_{L^\infty} \leqc t^{-\frac{1}{N+1}} \|F\|_{W^{1,1}} \|g\|_{W^{1,1}}.$$
Similarly, 
$$ |T_2| \leqc t^{-\frac{1}{N+1}}(\|\pa_y F\|_{L^1}\|g\|_{L^\infty}+ \|\pa_y g\|_{L^1}\|F\|_{L^\infty}) \leqc t^{-\frac{1}{N+1}}\|F\|_{W^{1,1}}\|g\|_{W^{1,1}}. $$
Lastly, using Lemma~\ref{key_lemma_2} we find 
\begin{align}\label{estimate:good}
    |T_1|&\le \|F\|_{L^\infty}\|g\|_{L^\infty} \int_{J} \left|\frac{\dee}{\dee y} \frac{1}{S_t(y)}\right| \, \dee y \leqc  t^{-\frac{1}{N+1}}\|F\|_{W^{1,1}}\|g\|_{W^{1,1}}.
\end{align}
The estimates above, together with \eqref{eq:IBPthreeterms}, yields
\begin{equation} \label{eq:IBPsecondterm}
    \bigg|\frac{i}{k} \int_{J}\frac{1}{S_t(y)}\pa_y (e^{-ik\varphi_t(y)})F(y)g(y) \, \d y \bigg| \leqc t^{-\frac{1}{N+1}} \|F\|_{W^{1,1}}\|g\|_{W^{1,1}}.
\end{equation}
Putting \eqref{eq:IBPterm1} and \eqref{eq:IBPsecondterm} into \eqref{eq:IBPsplit} completes the proof.
\end{proof}

We are now ready to prove Theorem~\ref{thm:main}.

\begin{proof}[Proof of Theorem~\ref{thm:main}]
For $t \ge t_\nu = \nu^{-p}$, the bound \eqref{eq:IBP_goal} is an immediate consequence of standard enhanced dissipation estimates. Indeed, it is well known that Assumption~\ref{assumption} implies
\begin{equation} \label{eq:ED}
\|\E[f_0 \circ \Phi_t]\|_{L^2} \leqc e^{-\nu^{\frac{N+1}{N+3}}t} \|f_0\|_{L^2}; 
\end{equation}
see, e.g., \cite{Villringer24,ABN22, wei2021diffusion, VictorJonathanKyle}. By the bound above and the continuous embedding $\ell_k^2 (W_y^{1,1})\hookrightarrow L^2$, we have
\begin{align}\label{eq:enhanced_IBP}
        \int_{\T^2}\E[f_{0}\circ \Phi_{t}(x,y)] \,   g(x,y) \, \dee x \, \dee y \leqc e^{-\nu^{\frac{N+1}{N+3}}t}\|f_0\|_{\ell_k^2 (W_y^{1,1})} \|g\|_{\ell_k^2 (W_y^{1,1})}.
\end{align}
The choice $p > \frac{N+1}{N+3}$ ensures that $e^{-\nu^{\frac{N+1}{N+3}}t} \leqc t^{-1}$ for $t \ge t_\nu$, which when combined with \eqref{eq:enhanced_IBP} implies \eqref{eq:IBP_goal}.

%It remains to show that the exponential decay in \eqref{eq:enhanced_IBP} is faster than $t^{-\frac{1}{N+1}}$. Let $M>0$ be a sufficiently large constant. We observe the chain of inequalities
%\begin{equation}\label{eq:decay}
  %  e^{-\nu^{\frac{N+1}{N+3}}t} = (\nu^{\frac{N+1}{N+3}}t)^{-M}(\nu^{\frac{N+1}{N+3}}t)^{M}e^{-\nu^{\frac{N+1}{N+3}}t} \leq C_M(\nu^{\frac{N+1}{N+3}}t)^{-M} \leq C_M t^{M(\frac{1}{p}\frac{N+1}{N+3}-1)} \lesssim t^{-\frac{1}{N+1}}.
%\end{equation}
%The final inequality holds by choosing $M$ large enough, which is permissible because our assumption $p>\frac{N+1}{N+3}$ guarantees that the exponent  $\frac{1}{p}\frac{N+1}{N+3}-1$ is strictly negative.

To complete the proof, it remains to prove \eqref{eq:IBP_goal} for $t\in [1,t_\nu]$. We first split the expectation between the good and bad noise realizations:
\begin{align}
  & \int_{\T^2}\E[f_{0}\circ \Phi_{t}(x,y)] \,  g(x,y) \, \dee x \, \dee y \\ 
  & \qquad =\int_{\Omega^c_\delta} \int_{\T^2}f_{0}\circ \Phi_{t}(x,y) \, g(x,y) \, \dee x \, \dee y \, \P(\dee \omega) +  \int_{\Omega_\delta} \int_{\T^2} := T_1 + T_2.\label{eq:estimate_IBP_prob}
\end{align}
To estimate $T_1$, we apply Lemma~\ref{lem:badnoise} and the Cauchy-Schwarz inequality to obtain
\begin{align}
    |T_1| &\le \P(\Omega_\delta^c) \|g\|_{L^2} \sup_{\omega \in \Omega_\delta^c} \|f_0 \circ \Phi_t\|_{L^2} \\ 
    &\leqc t^{-\frac{1}{N+1}} \|g\|_{L^2} \|f_0\|_{L^2} \leqc t^{-\frac{1}{N+1}} \|g\|_{\ell_k^2(W_y^{1,1})} \|f_0\|_{\ell_k^2(W_y^{1,1})}, \label{eq:ibp_estimate_bad}
\end{align}
where we have noted that the random flow map preserves Lebesgue measure for each $\omega \in \Omega$ and again applied the embedding $\ell_k^2 (W_y^{1,1})\hookrightarrow L^2$. 
For $T_2$, we rewrite the integral on the Fourier side using Parseval's identity and apply Lemma~\ref{lem:IBPkey}. First, from \eqref{eq:flowmap} and a direct computation we have 
\begin{align*}
    (f_{0}\circ\Phi_t)_k(y) = e^{ik\sqrt{2\nu}B_t}e^{-ik\varphi_t(y)} (f_{0})_k(t,y+\sqrt{2\nu}W_t),
\end{align*}
where here we use the notation introduced in Section~\ref{sec:results}. Define $F_k(t,y) = e^{ik\sqrt{2\nu}B_t} (f_{0})_k(t,y+\sqrt{2\nu}W_t)$ and note that $\|F_k(t,\cdot)\|_{W^{1,1}_y} \le C \|(f_0)_k\|_{W^{1,1}}$ for a constant $C > 0$ that does not depend on $k \in \Z$ or $t > 0$. Thus, by Parseval's theorem and Lemma~\ref{lem:IBPkey}, for any $\omega \in \Omega_\delta$ we have
\begin{align*}
&\left| \int_{\T^2}f_{0}\circ \Phi_{t}(x,y) g(x,y) \, \dee x \, \dee y \right| = \Bigg| \sum_{k \in \Z \setminus \{0\}} \int_{\T}e^{-ik\varphi_t(y)} F_k(t,y)g_k(t,y) \d y\Bigg| \\ 
& \qquad \qquad \leqc \sum_{k \in \Z \setminus \{0\}} t^{-\frac{1}{N+1}}\|F_k(t,\cdot)\|_{W^{1,1}_y} \|g_k\|_{W^{1,1}_y} \\ 
& \qquad \qquad \leqc \sum_{k \in \Z \setminus \{0\}}t^{-\frac{1}{N+1}}\|(f_0)_k\|_{W^{1,1}_y} \|g_k\|_{W^{1,1}_y} \le t^{-\frac{1}{N+1}}\|f_0\|_{\ell_k^2(W_{y}^{1,1})} \|g\|_{\ell_k^2 (W_y^{1,1})}.
\end{align*}
The above estimate implies that $|T_2| \leqc t^{-\frac{1}{N+1}}\|f_0\|_{\ell_k^2(W_{y}^{1,1})} \|g\|_{\ell_k^2 (W_y^{1,1})}$. Together with 
\eqref{eq:ibp_estimate_bad} and \eqref{eq:estimate_IBP_prob}, we have shown that 
$$ \int_{\T^2}\E[f_{0}\circ \Phi_{t}(x,y)] \,  g(x,y) \, \dee x \, \dee y \leqc t^{-\frac{1}{N+1}}\|f_0\|_{\ell_k^2(W_{y}^{1,1})} \|g\|_{\ell_k^2 (W_y^{1,1})}$$
for $1 \le t \le t_\nu$, which completes the proof.
\end{proof}

\section{Dynamical argument: Proof of Theorem~\ref{thm:main2}}
In this section, we prove Theorem~\ref{thm:main2}. By an approximation argument, it is sufficient to show that for any $f_0 \in C^\infty(\T^2)$ satisfying \eqref{eq:meanzero}, $g \in W^{1,\infty}(\T)$, and $x \in \T$ there holds
\begin{equation}\label{eq: estimate}
    \left|\int_\T \E [f_0\circ \Phi_t(x,y)] \, g(y) \, \d y\right| \leqc \langle t \rangle^{-\frac{1}{N+1}} \|f_0\|_{L^\infty_x (W^{1,\infty}_y)}\|g\|_{W^{1,\infty}(\T)}.
\end{equation}
Indeed, for general $f_0 \in L_x^\infty(W_y^{1,\infty})$ we may take a sequence of smooth functions $f_{0,n}$ with $f_{0,n} \to f_0$ in $L_x^\infty(W_y^{1,\infty})$ and use \eqref{eq: estimate} to deduce 
\begin{equation} \label{eq:passlimit}
    \|\E [f_{0,n} \circ \Phi_t]\|_{L_x^\infty(W_{y}^{-1,1})} \leqc \langle t \rangle^{-\frac{1}{N+1}} \|f_{0,n}\|_{L^\infty_x (W^{1,\infty}_y)}.
\end{equation}
Since $(x,y) \mapsto \Phi_t(x,y)$ is continuous for fixed $t\ge 0$ and $\omega \in \Omega$, we have 
$$ \|\E [f_{0,n} \circ \Phi_t] - \E [f_{0} \circ \Phi_t]\|_{L^\infty(\T^2)} \le \|f_{0,n} - f_0\|_{L^\infty(\T^2)}, $$
and hence $\E [f_{0,n} \circ \Phi_t] \to \E [f_{0} \circ \Phi_t]$ in $L^\infty(\T^2)$. We may therefore pass $n \to \infty$ in \eqref{eq:passlimit} to obtain the result of Theorem~\ref{thm:main2}. The restriction to smooth initial data is convenient simply so that $|f_0 \circ \Phi_t(x,y)| \le \|f_0\|_{L^\infty(\T^2)}$ holds pointwise over $(x,y) \in \T^2$.

For $t\geq t_\nu$, the bound \eqref{eq: estimate} follows as in the proof of Theorem~\ref{thm:main} by using the $L^\infty$ version of \eqref{eq:ED} established in \cite{VictorJonathanKyle} (see in particular \cite[Remark 2.2]{VictorJonathanKyle}). When $t \in [1,t_\nu]$, one can split the expectation in \eqref{eq: estimate} as in \eqref{eq:estimate_IBP_prob} and apply Lemma~\ref{lem:badnoise} to reduce the proof of \eqref{eq: estimate} to
\begin{equation} \label{eq:dynamicsgoal}
    \left|\int_\T f_0 \circ \Phi_t(x,y) \, g(y) \, \dee y \right| \leqc t^{-\frac{1}{N+1}} \|f_0\|_{L^\infty_x (W^{1,\infty}_y)}\|g\|_{W^{1,\infty}(\T)}, \qquad \forall \, \omega \in \Omega_\delta. 
\end{equation}
Our goal is thus to obtain the estimate \eqref{eq:dynamicsgoal} for $t \in [1,t_\nu]$ and smooth $f_0$. Throughout the remainder of the section, $t \in [1,t_\nu]$, $\delta$ is assumed small enough so that Lemma~\ref{key_lemma} applies, and $\omega \in \Omega_\delta$ is fixed.

Let $\{I_i\}_{i=1}^m$ be the collection of intervals guaranteed by Lemma~\ref{key_lemma} and set $J = \T \setminus \cup_{i=1}^m I_i$. Then, $J$ can be decomposed into a family of disjoint intervals $\{J_i\}_{i=1}^m$. The main step in the proof of Theorem~\ref{thm:main2} is estimating the integral in \eqref{eq:dynamicsgoal} when restricted to a given interval $J_i$. This estimate is provided by the following lemma.

\begin{lemma} \label{lem:dynamicskey}
    Let $\{J_i\}_{i=1}^m$ be as above. Then, for each $1 \le i \le m$ we have 
    \begin{equation} \label{eq:dynamicskey}
        \left|\int_{J_i} f_0 \circ \Phi_t(x,y) \, g(y) \, \dee y \right| \leqc t^{-\frac{1}{N+1}} \|f_0\|_{L^\infty_x (W^{1,\infty}_y)}\|g\|_{W^{1,\infty}(\T)}
    \end{equation}
\end{lemma}

We momentarily defer the proof Lemma~\ref{lem:dynamicskey} and use it to complete the proof of Theorem~\ref{thm:main2}.

\begin{proof}[Proof of Theorem~\ref{thm:main2}]
As discussed earlier, it suffices to prove \eqref{eq:dynamicsgoal}. We split the integral to obtain 
\begin{equation} \label{eq:1.2goal}
   \left|\int_{\T} f_0 \circ \Phi_{t}(x,y) \, g(y) \, \dee y\right| 
    \leq   \sum_{i=1}^m\left|\int_{I_i}  f_0 \circ \Phi_{t}(x,y) \, g(y) \, \dee y\right| + \left|\int_{J} f_0 \circ \Phi_{t}(x,y) \, g(y) \, \dee y\right|.
\end{equation}
For the first integral above, we apply Lemma~\ref{key_lemma} and 
$$ \sup_{y \in \T} |f_0 \circ \Phi_t(x,y)| \le \|f_0\|_{L^\infty(\T^2)}, $$
which holds since $f_0$ is smooth, to obtain
\begin{align}
     \sum_{i=1}^m\left|\int_{I_i}  f_0 \circ \Phi_{t}(x,y) \, g(y) \, \dee y\right| &\leq \sum_{i=1}^m|I_i|\|f_0\|_{L^\infty(\T^2)} \|g\|_{L^\infty(\T)} \\ 
     & \leqc \langle t \rangle^{-\frac{1}{N+1}}\|f_0\|_{L_x^\infty (W_{y}^{1,\infty})} \|g\|_{W^{1,\infty}(\T)}. \label{eq:dynamical_bad}
\end{align}
To estimate the second term in \eqref{eq:1.2goal}, we apply Lemma~\ref{lem:dynamicskey} to deduce
\begin{align}
    \left|\int_{J} f_0 \circ \Phi_{t}(x,y) g(x,y) \dee y\right| &\leq \sum_{i=1}^m \left|\int_{J_i} f_0 \circ \Phi_{t}(x,y) g(x,y) \dee y\right| \\
    & \leqc t^{-\frac{1}{N+1}} \|f_0\|_{L^\infty_x (W^{1,\infty}_y)}\|g\|_{W^{1,\infty}(\T)}. \label{eq:dynamical_good}
\end{align}
Using \eqref{eq:dynamical_bad} and \eqref{eq:dynamical_good} yields \eqref{eq:dynamicsgoal}, completing the proof.
\end{proof}

The remainder of this section is devoted to the proof of Lemma~\ref{lem:dynamicskey}. We fix a given interval $J_i \subseteq \T$ and denote it by $I$. For the sake of convenient notation, we extend the test function $g \in W^{1,\infty}(\T)$ to a function on $\T^2$ that is constant in $x$. We denote the extended function by $g(x,y)$.

The proof of Lemma~\ref{lem:dynamicskey} is based on the change of variable
\begin{equation}\label{eq: pushforward}
    \left|\int_{I} f_0 \circ \Phi_{t}(x,y) \, g(x,y) \, \dee y\right|  =\left| \int_{\Phi_{t}(I)} f_0 (x,y) \, g\circ \Phi^{-1}_{t}(x,y) J_t(y) \, \dee \m_s \right|.
\end{equation}
Here, $\m_s$ denotes the arc length measure, and $J_t(y)$ is the Jacobian factor, defined by
\begin{equation}\label{eq: stretching_factor}
   J_t(y):= \frac{1}{\sqrt{1+S_t^2(y-\sqrt{2\nu}W_t)}}.
\end{equation}
As discussed in Section \ref{sec:ideas}, we will establish estimates of the type in \eqref{eq:gammamean}. To this end, we first describe the geometry of the image curve $\Phi_t(I)$.

\begin{lemma}\label{lemma:gamma}
    The image curve $\Phi_{t}(I)$ can be decomposed as
    \[
    \Phi_{t}(I) = \left(\bigcup_{j=1}^M \gamma_j\right) \cup R,
    \]
    where each $\gamma_j$ is a smooth curve that cross $\T$ horizontally exactly once and $R$ is a remainder satisfying $m_s(R) \leqc 1$. Precisely, for each $1\leq j\leq M$, there exists a smooth, strictly monotone function $h_j(x): \T \to \T$ such that $\gamma_j$ can be realized as the graph
    \[
     \gamma_j = \big\{ (x, h_j(x)) \in \T^2 : x \in \T \big\}.
    \]
 Moreover,
  \begin{equation}\label{eq: h}
    h_j'(x) = -\frac{1}{S_t(h_j(x)-\sqrt{2\nu}W_t)} \ \ \ \text{and} \ \ \ |h_j'(x)|\leq t^{-\frac{1}{N+1}}.
\end{equation}
\end{lemma}

\begin{proof}
Let the horizontal coordinate of $I$ be given by $x_0 \in \T$ and consider the parameterized curve 
\begin{equation}
    \gamma: I \to \R^2, \quad \gamma(y) = \Big(x_0+\sqrt{2\nu}B_t - \int_0^tb(y+\sqrt{2\nu}W_s) \ \d s,y+\sqrt{2\nu}W_t\Big):= (X(y),Y(y)).
\end{equation}
Observe that $X'(y) = -S_t(y)$, which by Lemma~\ref{key_lemma} satisfies $|X'(y)| \gtrsim t^{\frac{1}{N+1}}$ and has a fixed sign for $y \in I$, implying that $X(y)$ is strictly monotone. Then, by the Inverse Function Theorem, $X$ has a smooth inverse function $X^{-1}: X(I) \to I$. We define
\[
h(x):X(I)\to \R, \quad h(x) = X^{-1}(x) + \sqrt{2\nu}W_t.
\]
Then, $\Phi_t(I) = \Pi(\{(x,h(x)): x \in X(I)\})$, where $\Pi: \R^2 \to \T^2$ denotes the canonical projection. We obtain the curves $\gamma_j$ by decomposing the interval $X(I) \subseteq \R$ into pieces of length one, taking the portion of the graph of $h$ over each piece, and then projecting back to $\T^2$. The equality in \eqref{eq: h} is immediate from the construction and the upper bound follows from Lemma~\ref{key_lemma}, since $h_j(x) - \sqrt{2\nu} W_t \in I$.

\end{proof}

With Lemma~\ref{lemma:gamma} in hand, we are now ready to prove Lemma~\ref{lem:dynamicskey}.

\begin{proof}[Proof of Lemma~\ref{lem:dynamicskey}]

Recall that we extended the test function $g$ in the statement of the lemma to a function on $\T^2$ that is independent of $x$. Our goal is then to prove \eqref{eq:dynamicskey} with the $W^{1,\infty}$ norm on the right-hand side taken over $\T^2$. We begin by applying Lemma~\ref{lemma:gamma} to decompose \eqref{eq: pushforward}:
\begin{align}\label{eq:estimate_main}
 \quad \ \left|\int_{\Phi_{t}(I)} f_0 (x,y) \, g\circ \Phi^{-1}_{t}(x,y) J_t(y) \, \dee \m_s\right| \leq 
\left|\int_{R} \right|  + \left|\sum_{j=1}^M \int_{\gamma_j}  \right| := T_1 + T_2.\numberthis
\end{align}
For $(x,y)\in \Phi_t(I)$, we have $y-\sqrt{2\nu}W_t \in I$, and so $|S_t(y-\sqrt{2\nu}W_t)|\gtrsim t^{\frac{1}{N+1}}$ by Lemma~\ref{key_lemma}. Therefore, $J_t(y) \lesssim t^{-\frac{1}{N+1}}$ and the remainder term satisfies
\begin{equation}\label{eq:remainder}
      T_1 \lesssim t^{-\frac{1}{N+1}} m_s(R) \|f_0\|_{L^\infty(\T^2)}\|g\|_{L^\infty(\T^2)} \leqc t^{-\frac{1}{N+1}}\|f_0\|_{L_x^\infty(W_y^{1,\infty})} \|g\|_{W^{1,\infty}(\T^2)}.
\end{equation}

To complete the proof, it remains to show that 
\begin{equation} \label{eq:T2goal}
    T_2 \leqc t^{-\frac{1}{N+1}}\|f_0\|_{L_x^\infty(W_y^{1,\infty})} \|g\|_{W^{1,\infty}(\T^2)}.
\end{equation}
For a fixed $\gamma_j$, we have
\begin{align*}
   & \quad \ \left|\int_{\gamma_j} f_0 (x,y) \, g\circ \Phi^{-1}_{t}(x,y) J_t(y) \, \dee \m_s \right| \leq \left| \int_{\gamma_j} f_0 (x,y) \Big(g\circ \Phi^{-1}_{t}(x,y)-g_j\Big) J_t(y) \ \dee \m_s \right| \\
    &+ \left|\int_{\gamma_j}f_0(x,y) \, g_j \, J_j \, \dee \m_s\right| + \left| \int_{\gamma_j} f_{0} (x,y) \, g_j \Big(J_t(y)-J_j\Big) \, \dee \m_s \right| :=|T_{2,1}| + |T_{2,2}| + |T_{2,3}|,
\end{align*} 
where here we have used the notation $g_j:= \fint_{\gamma_j} g \circ \Phi^{-1}_{t} \,  \dee\m_s$ and $J_j:=\fint_{\gamma_j}J_t(y) \, \dee \m_s$.
We will estimate each of the three terms above and obtain \eqref{eq:T2goal} by summing the results over $1 \le j \le M$. 

\textbf{Estimate of $T_{2,1}$}: 
We begin by observing that $g\circ \Phi_t^{-1}$ has small deviation from its mean on $\gamma_j$. For any $z \in \gamma_j$, we have
    \begin{align*}
        g\circ \Phi^{-1}_{t}(z) - g_j &= \fint_{\gamma_j} \Big(g\circ \Phi^{-1}_{t}(z)-g\circ \Phi^{-1}_{t}(z')\Big) \, \dee \m_s(z') \\
        &\leq \sup_{z_1,z_2 \in \Phi^{-1}_t(\gamma_j)} |g(z_1)-g(z_2)| \leq \|g\|_{W^{1,\infty}} \ |\Phi^{-1}_t (\gamma_j)|.
    \end{align*}
   This yields
   \begin{equation}\label{eq:T_21}
         |T_{2,1}|\lesssim \|f_0\|_{L^{\infty}} \|g\|_{W^{1,\infty}}\ |\Phi^{-1}_t (\gamma_j)| \int_{\gamma_j} J_t\ \dee \m_s \lesssim  t^{-\frac{1}{N+1}} \|f_0\|_{L^{\infty}_x(W^{1,\infty}_y)} \|g\|_{W^{1,\infty}} \int_{\gamma_j} J_t\ \dee \m_s. 
   \end{equation}
The last inequality follows from the estimate $|\Phi_t^{-1}(\gamma_j)|\lesssim t^{-\frac{1}{N+1}}$. To see this, observe that
    \begin{equation}\label{eq:stretch}
    |\Phi_t^{-1}(\gamma_j)| = \int_{\Phi_t^{-1}(\gamma_j)} \d y = \int_{\gamma_j} J_t \ \d
        \m_s \leq \m_s(\gamma_j) t^{-\frac{1}{N+1}}\lesssim t^{-\frac{1}{N+1}}. 
    \end{equation}
    Here we have used the fact that $J_t(y)\lesssim t^{-\frac{1}{N+1}}$ whenever $(x,y) \in \Phi_t(I)$.
    
\textbf{Estimate of $T_{2,2}$}: We leverage the explicit parameterization of $\gamma_j$ and the condition that $f$ has zero mean along $x$. It is clear that
    \begin{align*}
        |T_{2,2}|= \frac{|g_j| }{\m_s(\gamma_j)}\int_{\gamma_j} J_t\ \dee \m_s\left|\int_{\gamma_j} f_0(x,y) \ \dee \m_s \right|\lesssim \|g\|_{W^{1,\infty}} \int_{\gamma_j} J_t\ \dee \m_s  \left|\int_{\gamma_j} f_0(x,y) \ \dee \m_s\right|.
    \end{align*}
To estimate the last integral on the right-hand side above, we first use the parametrization of $\gamma_j$, Taylor's theorem, and \eqref{eq: h} to obtain
\begin{align*}
      \int_{\gamma_j} f_0(x,y) \, \dee \m_s &= \int_0^{1} f_0(x,h_j(x)) \sqrt{1+|h'_j(x)|^2} \, \dee x =\int_0^1 f_0(x,h_j(x))\, \dee x + O\big(t^{-\frac{2}{N+1}}\big)\|f_0\|_{L^\infty}.
\end{align*}
  Since $f_0$ satisfies \eqref{eq:meanzero}, we have
     \begin{align*}
         \left|\int_0^{1} f_0(x,h_j(x))\,  \dee x\right| &= \left|\int_0^{1} \Big(f_0(x,h_j(x))-f_0(x,h_j(0))\Big) \, \dee x\right|\\
          & \lesssim \|f_0\|_{L^{\infty}_x(W^{1,\infty}_y)}\ \sup_{x\in \T}|h_j'(x)| \lesssim t^{-\frac{1}{N+1}} \|f_0\|_{L^{\infty}_x(W^{1,\infty}_y)},
     \end{align*}
    where in the last line we again applied \eqref{eq: h}. Combining the estimates above, we have shown that
    \begin{equation}\label{eq:T22}
         |T_{2,2}| \lesssim  t^{-\frac{1}{N+1}}   \|f_0\|_{L^{\infty}_x(W^{1,\infty}_y)}\|g\|_{W^{1,\infty}}\int_{\gamma_j} J_t\ \dee \m_s.
    \end{equation}
   
\textbf{Estimate of $T_{2,3}$}: Let $r(x) = (1,h_j'(x))$ denote the tangent vector to $\gamma_j$. Observe that
     \begin{align}
          |T_{2,3}| =  \left|\int_{\gamma_j} f_0(x,y)\, g_j\, \Big(J_t(y)-J_j\Big) \, \dee \m_s\right|  &\le \|f_0\|_{L^{\infty}_x(W^{1,\infty}_y)}\|g\|_{W^{1,\infty}} \int_{\gamma_i}|J_t(y)-J_j| \ \dee \m_s \\
          &\leqc  \|f_0\|_{L^{\infty}_x(W^{1,\infty}_y)}\|g\|_{W^{1,\infty}} \int_0^1 |\grad J_t\big(h_j(x)\big)\cdot r(x)| \ \dee x, \label{eq:T_22}
     \end{align}
     where the second inequality follows by applying the fundamental theorem of line integrals to bound $|J_t(y) - J_j|$. We now estimate the integral in \eqref{eq:T_22}, which for convenience we denote by $T_{3}$. A direct computation shows that
     \begin{align*}
     (\grad J_t)(h_j(x)) &= \Big(0, \big(1+S_t^2(H_j(x))\big)^{-3/2}S_t\big(H_j(x)\big)S_t'\big(H_j(x)\big)\Big) \\ 
     & = J_t(h_j(x)) \left(0,\big(1+S_t^2(H_j(x))\big)^{-1}S_t\big(H_j(x)\big)S_t'\big(H_j(x)\big)\right),
     \end{align*}
     where $H_j(x) = h_j(x) - \sqrt{2\nu} W_t$. Since $h_j'(x) = -(S_t'(H_j(x)))^{-1}$ by \eqref{eq: h}, we then have
      \begin{align*}
         T_{3} = \int_0^1 \Bigg|\frac{S_t'\big(H_j(x)\big)}{1+S_t^2(H_j(x))}\Bigg| J_t(h_j(x)) \, \dee x &\leq \int_0^1  \Bigg|\frac{S_t'(H_j(x))}{1+S_t^2(H_j(x))}\Bigg| J_t(h_j(x)) \sqrt{1+|h'_j(x)|^2} \,  \dee x \\&=\int_{\gamma_j} \Bigg|\frac{S_t'(y-\sqrt{2\nu}W_t)}{1+S_t^2(y-\sqrt{2\nu}W_t)}\Bigg| J_t(y)  \, \dee \m_s.
      \end{align*}
Putting the bound above into \eqref{eq:T_22} and reversing the change of variables yields
    \begin{align}
        |T_{2,3}|&\lesssim \|f_0\|_{L^{\infty}_x(W^{1,\infty}_y)}\|g\|_{W^{1,\infty}} \int_{\gamma_j} \Bigg|\frac{S_t'(y-\sqrt{2\nu}W_t)}{1+S_t^2(y-\sqrt{2\nu}W_t)}\Bigg| J_t(y)  \, \dee \m_s \  \\ &= \|f_0\|_{L^{\infty}_x(W^{1,\infty}_y)}\|g\|_{W^{1,\infty}} \int_{\Phi^{-1}_t(\gamma_j)}\bigg| \frac{S_t'(y)}{1+S_t^2(y)} \bigg| \, \dee y. \label{eq:T_23}
    \end{align}

\textbf{Conclusion}. Summing \eqref{eq:T_23} over $j$ and applying Lemma~\ref{key_lemma_2} gives
 \begin{align*}
      \sum_{j=1}^M |T_{2,3}| &\lesssim \|f_0\|_{L^{\infty}_x(W^{1,\infty}_y)}\|g\|_{W^{1,\infty}} \sum_{j=1}^M \int_{\Phi^{-1}_t(\gamma_j)} \bigg|\frac{\dee}{\dee y}\left(\frac{1}{S_t(y)}\right)\bigg| \,  \dee y \\
      &= \|f_0\|_{L^{\infty}_x(W^{1,\infty}_y)}\|g\|_{W^{1,\infty}} \int_{I} \bigg|\frac{\dee}{\dee y}\left(\frac{1}{S_t(y)}\right)\bigg| \,  \dee y \lesssim t^{-\frac{1}{N+1}} \|f_0\|_{L^{\infty}_x(W^{1,\infty}_y)}\|g\|_{W^{1,\infty}} .
 \end{align*}
 Similarly for $T_{2,1}$ and $T_{2,2}$, we sum \eqref{eq:T_21} and \eqref{eq:T22} over $j$ to obtain
 \begin{align*}
     \sum_{j=1}^M (|T_{2,1}|+|T_{2,2}|) &\lesssim t^{-\frac{1}{N+1}} \|f_0\|_{L^{\infty}_x(W^{1,\infty}_y)}\|g\|_{W^{1,\infty}}  \sum_{j=1}^M\int_{\Phi^{-1}_t(\gamma_j)} J_t \, \d \m_s \\
     &\lesssim t^{-\frac{1}{N+1}} \|f_0\|_{L^{\infty}_x(W^{1,\infty}_y)}\|g\|_{W^{1,\infty}} \int_{I} \d y \lesssim t^{-\frac{1}{N+1}} \|f_0\|_{L^{\infty}_x(W^{1,\infty}_y)}\|g\|_{W^{1,\infty}} .
 \end{align*}
Then, combining these estimates, we establish \eqref{eq:T2goal}, which proves Lemma~\ref{lem:dynamicskey}.
\end{proof}

\bibliographystyle{abbrv}    
\bibliography{uniform_mixing}
\end{document}